\documentclass[12pt]{amsart} 
\IfFileExists{srcltx.sty}{\usepackage[active]{srcltx}}{}

\usepackage{amscd}
\usepackage{amsfonts,amssymb,latexsym}
\newcommand{\cE}{{\mathcal{E}}}
\newcommand{\cL}{{\mathcal{L}}}
\newcommand{\cM}{{\mathcal{M}}}
\newcommand{\cO}{{\mathcal{O}}}
\newcommand{\cU}{{\mathcal{U}}}
\newcommand{\cV}{{\mathcal{V}}}

\newcommand{\PP}{\mathbb{P}}
\newcommand{\ZZ}{\mathbb{Z}}
\newcommand{\CC}{\mathbb{C}}
\newcommand{\DD}{\mathbb{D}}

\newcommand{\fg}{\mathfrak{g}}

\newcommand{\kbar}{\overline{k}}
\newcommand{\Ohat}{\widehat{\cO}}
\newcommand{\tG}{\widetilde{G}}
\newcommand{\tfg}{\widetilde{\fg}}
\newcommand{\Gm}{\mathbb{G}_{\mathrm{m}}}

\DeclareMathOperator{\Spec}{Spec}
\DeclareMathOperator{\Pic}{Pic}
\DeclareMathOperator{\Hom}{Hom}
\DeclareMathOperator{\Aut}{Aut}
\DeclareMathOperator{\PGL}{PGL}
\DeclareMathOperator{\SL}{SL}
\DeclareMathOperator{\ad}{ad}
\DeclareMathOperator{\rank}{rank}

\newcommand{\dual}{\mathrm{dual}}
\newcommand{\univ}{\mathrm{univ}}
\newcommand{\stab}{\mathrm{s}}
\newcommand{\regstab}{\mathrm{rs}}
\newcommand{\Poinc}{\mathrm{Poinc}}

\newcommand{\too}{\longrightarrow}

\newtheorem{proposition}{Proposition}[section]
\newtheorem{theorem}[proposition]{Theorem}
\newtheorem*{theorem*}{Theorem}
\newtheorem{corollary}[proposition]{Corollary}

\theoremstyle{remark}
\newtheorem{remark}[proposition]{Remark}

\title[Stability of the Poincar\'e bundle]{Stability of the Poincar\'e bundle}

\author[I. Biswas]{Indranil Biswas}

\address{School of Mathematics, Tata Institute of Fundamental
Research, Homi Bhabha Road, Bombay 400005, India}

\email{indranil@math.tifr.res.in}

\author[T. L. G\'omez]{Tom\'as L. G\'omez}

\address{Instituto de Ciencias Matem\'aticas (CSIC-UAM-UC3M-UCM),
Nicol\'as Cabrera 15, Campus Cantoblanco UAM, 28049 Madrid, Spain}

\email{tomas.gomez@icmat.es}

\author[N. Hoffmann]{Norbert Hoffmann}

\address{Department of Mathematics and Computer Studies, Mary
Immaculate College, South Circular Road, Limerick, Ireland}

\email{norbert.hoffmann@mic.ul.ie}



\date{}

\begin{document}

\begin{abstract}
Let $X$ be an irreducible smooth projective curve, of genus at least two, over an 
algebraically closed field $k$. Let $\cM^d_G$ denote the moduli stack of principal 
$G$--bundles over $X$ of fixed topological type $d \in \pi_1(G)$, where $G$ is any almost simple 
affine algebraic group over $k$. We prove that the universal bundle over $X 
\times \cM^d_G$ is stable with respect to any polarization on $X \times \cM^d_G$. A 
similar result is proved for the Poincar\'e adjoint bundle over $X \times M_G^{d, 
\regstab}$, where $M_G^{d, \regstab}$ is the coarse moduli space of regularly stable 
principal $G$--bundles over $X$ of fixed topological type $d$.
\end{abstract}

\maketitle

\section{Introduction}

Let $X$ be a compact connected Riemann surface of genus $g_X \geq 2$.
Fix an integer $r \geq 2$ and a holomorphic line bundle $L$ over $X$ of degree $d$ such that $r$ is coprime to $d$.
Let $M_{r, L}$ denote the coarse moduli space of all stable vector bundles $E$ over $X$ with $\rank( E) = r$ and $\bigwedge^r E \cong L$.
A vector bundle over $X \times M_{r, L}$ is called a \emph{Poincar\'{e} bundle} if its restriction to each closed point $[E] \in M_{r, L}$
is isomorphic to $E$. 
It is known that Poincar\'e bundles over $X \times M_{r, L}$ exist,
and that any two of them differ by tensoring with a line bundle pulled back from $M_{r, L}$.
Balaji, Brambila-Paz and Newstead proved in \cite{BBN} that each Poincar\'e bundle over $X \times M_{r, L}$ is stable
with respect to any polarization on $X \times M_{r, L}$. This result allows to use Poincar\'e bundles to study
moduli spaces of vector bundles on the smooth projective varieties $M_{r, L}$ and $X \times M_{r, L}$, as constructed in \cite{Ma}.
It also provides an interesting metric on the Poincar\'e bundle via the Donaldson-Uhlenbeck-Yau correspondence \cite{Do,UY}.

The same question can be asked more generally for moduli spaces of principal $G$--bundles over $X$.
These moduli spaces are no longer smooth projective, and Poincar\'e $G$--bundles need not exist \cite{BH3}.
But over the open locus of regularly stable $G$--bundles, a Poincar\'e $G^{\ad}$--bundle always exists,
where $G^{\ad}$ denotes the quotient of $G$ modulo its center,
and the question whether it is stable still makes sense. For orthogonal and symplectic bundles, this stability is proved in \cite{BG}.

Now let $X$ be an irreducible smooth projective curve of genus $g_X \geq 2$ over an algebraically closed field $k$.
Let $G$ be a smooth connected almost simple algebraic group over $k$.
Let $\cM^d_G$ be a connected component of the moduli stack of principal $G$--bundles over $X$;
these connected components are indexed by the elements $d \in \pi_1( G)$.
From this point of view, we approach the above stability question in this paper.
In Section \ref{sec:stack}, we prove our main result, Theorem \ref{thm1}, which states the following:
\begin{theorem*}
The universal principal $G$--bundle over $X \times \cM^d_G$ is stable with respect to any polarization on $X \times \cM^d_G$.
\end{theorem*}
The proof uses the description of $\cM^d_G$ provided by \cite{KNR,BL,BLS},
based on the uniformization theorem of Drinfeld and Simpson \cite[Theorem 3]{DS}.
We first prove as Proposition \ref{prop-s} that the restriction of the universal principal $G$--bundle
to the slice $\{x\} \times \cM^d_G$ is semistable for any point $x$ on $X$. 
From this we deduce Theorem \ref{thm1}.

Section \ref{sec:coarse} deals with consequences concerning the coarse moduli scheme $M^d_G$. In particular, 
we prove as Corollary \ref{cor:stab} that the Poincar\'{e} $G^{\ad}$--bundle is stable with respect to any polarization.
Along the way, we again obtain that the restriction of this Poincar\'{e} $G^{\ad}$--bundle to the slice given by any point $x$ on $X$ is semistable. 

\subsection*{Acknowledgements}

The first author is supported by a J. C. Bose Fellowship.
The second author acknowledges funding from the Spanish 
MICINN (grants MTM2016-79400-P, PID2019-108936GB-C21,
and ICMAT Severo Ochoa projects
SEV-2015-0554 and CEX2019-000904-S), 
the 7th European Union Framework Programme
(Marie Curie IRSES grant 612534 project MODULI) and CSIC
(2019AEP151 and \textit{Ayuda extraordinaria a Centros de Excelencia 
Severo Ochoa} 20205CEX001).
The third author was supported by Mary Immaculate College Limerick through the PLOA sabbatical programme.
He thanks the Tata Institute of Fundamental Research in Bombay for its hospitality. 

\section{Stability over the moduli stack} \label{sec:stack}

Let $X$ be an irreducible smooth projective curve over an algebraically closed field $k$.
In this section, we allow the base field $k$ to have arbitrary characteristic.

Let $G$ be a smooth connected reductive algebraic group over $k$.
Let $\cM_G$ denote the moduli stack of principal $G$--bundles over $X$.
This stack is smooth over $k$, and its connected components $\cM_G^d$ are indexed by
the elements
\begin{equation*}
d \in \pi_1(G) := \Lambda_{T_G}/\Lambda_{\rm coroots}
\end{equation*}
according to \cite[Proposition 1.3]{BLS} and \cite[Theorem 5.8]{Ho}.
Here $\Lambda_{T_G} = \Hom( \Gm, T_G)$
is the cocharacter lattice of a maximal torus $T_G \subseteq G$,
and $\Lambda_{\rm coroots} \subseteq \Lambda_{T_G}$ is the coroot lattice.
If $k = \CC$, then $\pi_1(G)$ coincides with the topological fundamental group.

From now on, we assume that $G$ is almost simple. There is a natural homomorphism
\begin{equation} \label{eq:charge}
 \Pic( \cM_G^d) \too \ZZ,
\end{equation}
called \emph{central charge}, whose kernel and cokernel are both finite.
Its definition will be recalled in the proof of Proposition \ref{prop-s} below.
This central charge has been constructed in \cite[Theorem 2.4]{KN} and \cite[Proposition 1.5]{BLS}
for the case $k = \CC$, and in \cite[Theorem 17]{Fa} and \cite[Theorem 5.3.1]{BH} for arbitrary characteristic.

Let $\cL$ be a line bundle over an open substack $\cU \subseteq \cM_G^d$.
We say that $\cU$ is \emph{big} if its complement has codimension at least $2$.
Then $\cL$ extends uniquely to $\cM_G^d$ by \cite[Lemma 7.3]{BH3},
and we define $\deg \cL$ to be the central charge of this unique extension.
For a vector bundle $\cV$ of rank $r$ over $\cU$, we define $\deg \cV$ to be $\deg( \bigwedge^r \cV)$.

Let $\cE \too \cM_G^d$ be a principal $G$--bundle. We say that $\cE$
is \emph{stable} (respectively, \emph{semistable}) if for every reduction
$\cE \supset \cE_P \too \cU$ to a parabolic subgroup $P \subset G$
over a big open substack $\cU \subseteq \cM_G^d$, and every
strictly dominant character $\chi: P \too \Gm$ trivial on
the center of $G$, the associated line bundle $\cE_P(\chi)$ over $\cU$ satisfies
\begin{equation*}
 \deg \cE_P(\chi) < 0 \quad (\text{respectively, } \deg \cE_P(\chi) \leq 0).
\end{equation*}
Note that no choice of a polarization is needed here; it is given by the central charge.

For a principal $G$--bundle over $X \times \cM_G^d$, stability and semistability can be defined similarly,
but only after the choice of a polarization. For that, consider homorphisms
\begin{equation*}
\Pic( X \times \cM_G^d) \too \ZZ
\end{equation*}
that are locally constant in flat families of line bundles. Since $\Pic( \cM_G^d)$ is discrete, 
any such homomorphism is a linear combination of the degree on $X$ and the
central charge on $\cM_G^d$; we call it a polarization if it is a positive linear combination.
In the above definition of stability, the central charge can be replaced by any such polarization.

Let $\cE^{\univ} \too X \times \cM_G$ be the universal principal $G$--bundle, and let $x \in X$ be a closed point.
The principal bundle over $\cM_G$ obtained by restricting $\cE^{\univ}$ to the slice $\{x\} \times \cM_G$,
and also its further restriction to $\cM_G^d$, will both be denoted by $\cE^{\univ}_x$.

\begin{proposition}\label{prop-s}
 Let $G$ be a smooth connected almost simple algebraic group.
 Let $\cM_G^d$ be a connected component of the moduli stack of principal $G$--bundles over $X$.
 Then the principal $G$--bundle $\cE^{\univ}_x$ over $\cM_G^d$ is semistable.
\end{proposition}

\begin{proof}
 Choose a closed point $y \in X \setminus \{x\}$, and an isomorphism $\Ohat_{X, y} \cong k[[t]]$.
 Recall, e.\,g.\ from \cite{Fa}, that the \emph{loop group} $LG$ is an ind-scheme over $k$ with
 \begin{equation*}
 LG( k)\, = \,G( k((t))),
 \end{equation*}
  the \emph{positive loop group} is the sub-group scheme $L^+ G \subseteq LG$ given by
  \begin{equation*}
  L^+G( k) = G( k[[t]]),
  \end{equation*}
  and the \emph{affine Grassmannian} of $G$ is the ind-projective variety
  \begin{equation*}
    Q_G := LG/L^+ G.
  \end{equation*}

  The glueing procedure of \cite[Definition 1.4]{KNR} defines a morphism
  \begin{equation*}
    q_y: Q_G \too \cM_G
  \end{equation*}
  that sends the class of $\gamma \in LG$ to the trivial $G$--bundles over the disc $\DD := \Spec \Ohat_{X, y}$
  and over $X \setminus \{y\}$, glued by the automorphism $\gamma$ of the trivial $G$--bundle over $\DD \setminus \{y\}$.
  This morphism $q_y$ is well-defined because multiplication by $L^+ G$ from the right can be compensated by changing the trivialization over $\DD$.
  In particular, the pullback of $\cE^{\univ}$ from $X \times \cM_G$ to $X \times Q_G$ still comes with a trivialization over $X \setminus \{y\}$.
  Since $x \neq y$, it follows that the pullback of $\cE^{\univ}_x$ from $\cM_G$ to $Q_G$ is trivial as well.

  Now choose a lift of $d \in \pi_1( G)$ to a cocharacter $\delta: \Gm \too T_G \subseteq G$.
  We denote by $t^\delta \in LG(k )$ the image of the point $t \in L\Gm( k) = k((t))^*$ under $\delta$. Let
  \begin{equation*}
    \pi: \tG \too G
  \end{equation*}
  be the universal cover of $G$. The choice of $\delta$ implies that the composition
  \begin{equation*}
    Q_{\tG} \xrightarrow{t^{\delta} \cdot \pi ( \_)} Q_G \xrightarrow{q_y} \cM_G
  \end{equation*}
  maps to the component $\cM_G^d$ of $\cM_G$. This defines a morphism 
  \begin{equation*}
    q_y^{\delta}: Q_{\tG} \too \cM_G^d
  \end{equation*}
  which already appears in \cite[Proposition 1.5]{BLS}.
  Due to the uniformization theorem of Drinfeld and Simpson \cite{DS}, this morphism $q_y^{\delta}$ is surjective.

  Let $\tfg_{\alpha}$ denote the root space in the Lie algebra $\tfg$ of $\tG$ corresponding to a root $\alpha$.
  We denote by $a = 1 - \alpha$ the affine root of $L\tG$ corresponding to the root space
  \begin{equation*}
    t \cdot \tfg_{-\alpha} \subset L \tfg := k[t, t^{-1}] \otimes_k \tfg,
  \end{equation*}
  and by $U_a \subset L\tG$ the corresponding copy of the additive group over $k$. Let
  \begin{equation*}
    i_a: \SL_2 \too L\tG
  \end{equation*}
  denote the embedding whose image is generated by $U_a$ and $U_{-a}$. Then $i_a^{-1}( L^+ \tG)$ is a subgroup of triangular matrices in $\SL_2$.
  Therefore $i_a$ induces an embedding
  \begin{equation*}
    j_a: \PP^1 \too Q_{\tG}
  \end{equation*}
  whose image is a simple example of a Schubert variety in $Q_{\tG}$. The induced map
  \begin{equation*}
    (j_a)^*: \Pic( Q_{\tG}) \too \Pic( \PP^1) \cong \ZZ
  \end{equation*}
  is an isomorphism according to \cite[Proposition 2.3]{KNR} and \cite[Corollary 12]{Fa}.

  We remark that the composition
  \begin{equation*}
    \PP^1 \xrightarrow{j_a} Q_{\tG} \xrightarrow{q_y^{\delta}} \cM_G^d
  \end{equation*}
  is a generalization to principal $G$-bundles of the Hecke lines in the moduli space of vector bundles
  introduced in \cite[section 4, p.\ 397]{NR}. For instance, if $G=\SL_2$ and $\delta$ is the trivial cocharacter,
  then this $\PP^1$-family of principal $\SL_2$-bundles over $X$ is the $\PP^1$-family of all locally free sheaves
  $E$ of rank $2$ over $X$ such that
  \begin{equation*}
    \cO_X( -y) \oplus \cO_X \subsetneqq E \subsetneqq \cO_X \oplus \cO_X( y).
  \end{equation*}

  The central charge homomorphism in \eqref{eq:charge} is by definition the composition
  \begin{equation*}
    \Pic( \cM_G^d) \xrightarrow{(q_y^{\delta})^*} \Pic( Q_{\tG}) \xrightarrow{(j_a)^*} \Pic( \PP^1) \xrightarrow{\deg} \ZZ.
  \end{equation*}
  This homomorphism does not depend on the choices made in its construction.

  Choosing an embedding $G \hookrightarrow \SL_N$, it is easy to see that
  \begin{equation*}
    H^0( Q_{\tG}, L) \neq 0    
  \end{equation*}
  holds for \emph{some} $L \in \Pic( Q_{\tG})$ with $\deg( L_{|\PP^1}) > 0$; indeed, it holds for \emph{all} such $L$ according to \cite[Theorem 7]{Fa}.
  Since $Q_{\tG}$ is an inductive limit of reduced and irreducible projective Schubert varieties, we can conclude that
  \begin{equation*}
    H^0( Q_{\tG}, L) = 0
  \end{equation*}
  holds for all $L \in \Pic( Q_{\tG})$ with $\deg( L_{|\PP^1}) < 0$.

  Let $\cE^{\univ}_x \supset \cE_P \too \cU$ be a reduction of $\cE^{\univ}_x$ to a parabolic subgroup $P \subset G$
  over a big open substack $\cU \subseteq \cM_G^d$. By \cite[Remark 2.2]{Ra}, it suffices to prove
  \begin{equation*}
    \deg \ad( \cE_P) \leq 0
  \end{equation*}
  for the adjoint vector bundle $\ad( \cE_P) \too \cU$. This degree is by definition the central charge of the unique line bundle $\cL \too \cM_G^d$
  that extends the top exterior power
  \begin{equation*}
    \bigwedge\nolimits^{\dim( P)} \ad( \cE_P) \too \cU.
  \end{equation*}
  The inclusion $\ad( \cE_P) \subset \ad( \cE^{\univ}_x)$ induces an embedding of $\cL_{|\cU}$ into the vector bundle
  \begin{equation*} \label{eq:det_ad}
    \cV := \bigwedge\nolimits^{\dim P} \ad( \cE^{\univ}_x) \too \cM_G.
  \end{equation*}
  This embedding over $\cU$ extends to a generically injective morphism from $\cL$ to $\cV$ over $\cM_G^d$ due to Hartogs' theorem.
  By pullback, we obtain a nonzero morphism from
  \begin{equation*}
    L := (q_y^{\delta})^* \cL \too Q_{\tG}
  \end{equation*}
  to the vector bundle $(q_y^{\delta})^* \cV \too Q_{\tG}$.
  However, this vector bundle is trivial, since the principal $G$--bundle $(q_y)^* \cE^{\univ}_x \too Q_G$ is trivial. Therefore,
  \begin{equation*}
    H^0( Q_{\tG}, L^{\dual}) \neq 0,
  \end{equation*}
  and hence $\deg( L_{|\PP^1}) \leq 0$. But $\deg( L_{|\PP^1})$ is by definition the central charge of $\cL$.
\end{proof}

\begin{theorem}\label{thm1}
  If the curve $X$ has genus $g_X \geq 2$, then the universal bundle $\cE^{\univ}$
  over $X \times \cM_G^d$ is stable with respect to any polarization on $X\times \cM_G^d$.
\end{theorem}
\begin{proof}
  The restriction of $\cE^{\univ}$ to any point in $X$ is semistable by Proposition \ref{prop-s}.
  The restriction of $\cE^{\univ}$ to a general point in $\cM_G^d$ is stable,
  because the open locus $\cM_G^d \supset \cM^{d,\stab}_G$ of stable $G$--bundles over $X$ is nonempty for $g_X \geq 2$.
  These two facts imply the theorem, as the proof of \cite[Lemma 2.2]{BBN} shows.
\end{proof}

\begin{remark}
  If $G$ is only semisimple, then its universal cover $\tG$ is a product of $s \geq 2$ almost simple factors. Each of them defines a central charge,
  and their direct sum
  \begin{equation*}
    \Pic( \cM_G^d) \too \ZZ^s
  \end{equation*}
  again has finite kernel and cokernel. Proposition \ref{prop-s} remains true in this generality, now with respect to any polarization,
  because $H^0( Q_{\tG}, L) = 0$ for every line bundle $L$ over $Q_{\tG}$ with at least one negative central charge.
  Consequently, Theorem \ref{thm1}, and the two corollaries in the next section, all generalize to the semisimple case as well.
\end{remark}

\section{Stability over the coarse moduli space} \label{sec:coarse}

In this section, we assume that the base field is $k = \CC$.
Let $M_G$ denote the coarse moduli space of semistable principal $G$--bundles over $X$. 
Its connected components $M_G^d$ are irreducible normal projective varieties, indexed by $d \in \pi_1( G)$.

A principal $G$--bundle $E$ over $X$ is called \emph{regularly stable} if $E$ is stable and $\Aut( E) = Z_G$, the center of $G$.
We denote by $M_G^{d, \regstab} \subset M_G^d$
the open locus of regularly stable principal $G$--bundles. The quotient $G/Z_G$ will be denoted by $G^{\ad}$.
The appropriate restriction of the principal $G^{\ad}$--bundle
$\cE^{\univ}/Z_G$ over $X \times \cM_G^d$
 descends to a principal $G^{\ad}$--bundle
\begin{equation*}
 \cE^{\Poinc} \too X \times M_G^{d, \regstab}.
\end{equation*}
Let $\cE^{\Poinc}_x$ denote the restriction of $\cE^{\Poinc}$ to $\{x\} \times M_G^{d, \regstab}$ for a closed point $x \in X$.

From now on, we again assume that $G$ is almost simple.
We also assume that the curve $X$ has genus $g_X \geq 2$, and even $g_X \geq 3$ if $G^{\ad} \cong \PGL_2$.
Then $M^{d,\regstab}_G$ is the smooth locus of $M^d_G$ by \cite[Corollary 3.4]{BH2}.
In particular, the complement $M^d_G \setminus M^{d,\regstab}_G$ has codimension at least two.
Therefore, $\Pic( M_G^{d, \regstab})$ embeds as a subgroup into $\Pic( \cM_G^d)$ (cf. \cite[\S (13.1)]{BLS}), and the central charge restricts to a natural homomorphism
\begin{equation*}
\Pic( M_G^{d, \regstab}) \,\too\, \ZZ
\end{equation*}
whose kernel and cokernel are still both finite. This again provides a polarization on $M_G^{d, \regstab}$,
so there is no need to choose one in dealing with stability over $M_G^{d, \regstab}$.

As direct consequences of Proposition \ref{prop-s} and Theorem \ref{thm1}, we obtain the following.
\begin{corollary} \label{cor:semistab}
The principal $G^{\ad}$--bundle $\cE^{\Poinc}_x$ over $M_G^{d, \regstab}$ is semistable.
\end{corollary}

\begin{corollary} \label{cor:stab}
The Poincar\'{e} adjoint bundle $\cE^{\Poinc}$ over $X \times M_G^{d, \regstab}$ is stable with respect to any polarization on $X \times M^{d,\regstab}_G$.
\end{corollary}

\begin{remark}
If the base field $k = \kbar$ has characteristic $p > 0$, then Corollary \ref{cor:semistab} and Corollary \ref{cor:stab}
remain valid at least for $g_X \geq 4$, because the locus of regularly stable principal $G$--bundles $E$ over $X$
is still open by \cite[Proposition 2.3]{BH3}, and its complement still has codimension at least two by \cite[Theorem 2.5]{BH3}. 
\end{remark}

\end{document}